\documentclass[12pt]{article}

\usepackage[verbose]{geometry}
\input{PreludeA}

\usepackage{titling}
\usepackage{fancyhdr}
\hangsecnum

\title{From fractions to complete Segal spaces}
\author{Zhen~Lin Low and Aaron Mazel-Gee}
\date{May 15, 2015}

\begin{document}

\maketitle
\footpar{\textsc{E-mail addresses}: \texttt{Z.L.Low@dpmms.cam.ac.uk} and  \texttt{aaron@math.berkeley.edu}}

\begin{abstract}
We show that the Rezk classification diagram of a relative category admitting a homotopical version of the two-sided calculus of fractions is a Segal space up to Reedy-fibrant replacement. This generalizes the result of Rezk and Bergner on the classification diagram of a closed model category, as well as the result of Barwick and Kan on the classification diagram of a partial model category.
\end{abstract}

% !TEX root = FromFractionsToCSS.tex
\section{Introduction}

It is now commonly understood that one has a (generalized) homotopy theory whenever one has relative category, \ie, a category equipped with a subcategory of distinguished morphisms to be thought of as weak equivalences. The name `relative category' is due to \citet{Barwick-Kan:2012a}, but the idea already appeared in a paper of \citet{Dwyer-Kan:1980a}, in which they defined the simplicial localization of a relative category. Simplicial localization should be regarded as a homotopy-theoretic refinement of the usual procedure of freely inverting weak equivalences in a relative category, where instead of adjoining on-the-nose inverses for weak equivalences, one adjoins up-to-homotopy inverses. Of course, in doing so, one ends up also adjoining homotopies to the original category, which is one reason why we say the result of simplicial localization is a homotopy theory.

One particularly elegant reification of the concept of `homotopy theory' is the notion of `complete Segal space' introduced by \citet{Rezk:2001}: these are simplicial spaces (or more accurately, bisimplicial sets) that are Reedy-fibrant and satisfy certain conditions. These should be regarded as a homotopy-theoretic version of categories\hairspace ---\hairspace in other words, as $\tuple{\infty, 1}$-categories. In \opcit, Rezk defined for any relative category $\mathcal{C}$ the classification diagram $\Nv{\mathcal{C}}_{\bullet}$, a simplicial space whose $n$-th level classifies the weak equivalence classes of composable chains of morphisms in $\mathcal{C}$ of length $n$, and he proved the following result:

\begin{thm*}
Let $\mathcal{M}$ be a simplicial model category. Then any Reedy-fibrant replacement $\widehat{\Nv{\mathcal{M}}}_{\bullet}$ of the classification diagram $\Nv{\mathcal{M}}_{\bullet}$ is a complete Segal space, and moreover the hom-spaces of $\widehat{\Nv{\mathcal{M}}}_{\bullet}$ agree with the homotopy function complexes of $\mathcal{M}$ up to weak homotopy equivalence.
\end{thm*}

As it turns out, Rezk's result holds more generally. \citet[\Sect 6]{Bergner:2009} proved the case where $\mathcal{M}$ is a model category (\ie, not necessarily simplicial), without using functorial factorizations. On the other hand, \citet[\Sect 3]{Barwick-Kan:2011} verified the case where $\mathcal{M}$ is a partial model category. (The notion of partial model category is a generalization of the notion of model category with functorial factorizations.) The main result of this paper is a generalization of both results: we do not require functorial factorizations (as Barwick and Kan do), and we do not require lifting properties or (co)completeness (as Bergner does).

In fact, our result relies on very few assumptions. Let $\mathcal{C}$ be a relative category. By thinking geometrically, one sees that if the  Rezk classification diagram $\Nv{\mathcal{C}}_{\bullet}$ is Reedy weakly equivalent to a Segal space, then it must be the case that every morphism in $\Ho \mathcal{C}$ admits a factorization of the form
\[
\inv{w}_n \circ f_n \circ \cdots \circ \inv{w}_1 \circ f_1 \circ \inv{w}_0
\]
where $w_0, \ldots, w_n$ are weak equivalences in $\mathcal{C}$ and $f_1, \ldots, f_n$ are morphisms in $\mathcal{C}$, of which at most one is \emph{not} a weak equivalence. Thus one might expect that for the Rezk classification diagram $\Nv{\mathcal{C}}_{\bullet}$ to be a complete Segal space up to Reedy-fibrant replacement, it suffices that $\mathcal{C}$ be saturated in the sense of \citep[\Sect 8]{DHKS} and admit a suitable three-arrow calculus\hairspace ---\hairspace namely the homotopy calculus of fractions in the sense of \citet{Dwyer-Kan:1980b}. This is precisely what we will show. 

Aside from the Dwyer--Kan homotopy calculus of fractions, the key technical result we use is a version of Quillen's Theorem B for recognizing homotopy pullback diagrams. In this respect, our proof differs from both that of \citet{Bergner:2009}, which uses the classifying complex of the simplicial monoid of weak equivalences, and that of \citet {Barwick-Kan:2011}, which uses a sophisticated generalization of Quillen's Theorem B.

This paper is organized as follows:
\begin{itemize}
\item In \Sect 1, we recall how to do homotopy theory with categories \`a la \citet[\Sect 1]{Quillen:1973a}.

\item In \Sect 2, we set up notation and basic results for working with zigzags in relative categories.

\item In \Sect 3, we prove the main result.
\end{itemize}

\subsection*{Conventions}

\begin{itemize}
\item By `natural number' we mean a non-negative integer: $0, 1, 2, \ldots$

\item By `model category' we mean a closed model category in the sense of \citet{Quillen:1967}, \ie, we do \emph{not} require functorial factorizations and we only require \emph{finite} limits and colimits.

\item To avoid set-theoretic difficulties, we will focus on small categories, \ie, categories that only have a set of objects and morphisms rather than a proper class. This is no real restriction under the assumption of a suitable universe axiom.

\item We treat $\cat{\Cat}$ as an ordinary category: so for example, `pullback' refers to the strict notion.
\end{itemize}

\subsection*{Acknowledgements}

This collaboration would not have happened without the `Homotopy Theory' chat room on MathOverflow.

The first-named author gratefully acknowledges financial support from the Cambridge Commonwealth, European and International Trust and the Department of Pure Mathematics and Mathematical Statistics. The second-named author gratefully acknowledges financial support from UC Berkeley's Geometry and Topology RTG grant, which is part of NSF grant DMS-0838703. The authors thank an anonymous referee for helpful comments and suggestions.

\section{Homotopy theory with categories}

Every category $\mathcal{C}$ has an associated simplicial set $\nv{\mathcal{C}}$, called its \emph{nerve}, and this construction assembles into a functor $\nv : \cat{\Cat} \to \cat{\SSet}$.  This allows us to think of a category as being a ``presentation of a space''.  In this section, we recall some of the basic definitions and results concerning the manipulation of categories in this capacity.

\begin{dfn}
A \strong{weak homotopy equivalence} of categories is a functor $F : \mathcal{C} \to \mathcal{D}$ such that the morphism $\nv{F} : \nv{\mathcal{C}} \to \nv{\mathcal{D}}$ of nerves is a weak homotopy equivalence of simplicial sets.
\end{dfn}

\begin{lem}
\label{lem:homotopy.equivalences.of.categories}
Let $F : \mathcal{C} \to \mathcal{D}$ and $G : \mathcal{D} \to \mathcal{C}$ be functors. If there exist zigzags of natural transformations between $\id_\mathcal{C}$ and $G F$ and between $F G$ and $\id_\mathcal{D}$, then both $F$ and $G$ are weak homotopy equivalences of categories.
\end{lem}
\begin{proof}
As natural transformations translate to homotopies of morphisms of nerves, the hypothesis implies that the morphisms $\nv{F} : \nv{\mathcal{C}} \to \nv{\mathcal{D}}$ and $\nv{G} : \nv{\mathcal{D}} \to \nv{\mathcal{C}}$ are homotopy equivalences of simplicial sets, and hence weak homotopy equivalences \emph{a fortiori}.
\end{proof}

\begin{dfn}
A \strong{homotopy pullback diagram} of categories is a commutative square that $\nv : \cat{\Cat} \to \cat{\SSet}$ sends to a homotopy pullback diagram of simplicial sets.
\end{dfn}

Quillen's Theorem B gives us a way of recognizing homotopy pullback diagrams. Let us first recall the following definition:

\begin{dfn}
A \strong{Grothendieck fibration} is a functor $P : \mathcal{E} \to \mathcal{B}$ such that, for every object $E$ in $\mathcal{E}$ and every morphism $f : B' \to P \argp{E}$ in $\mathcal{B}$, there exist an object $f^* E$ in $\mathcal{E}$ and a morphism $g : f^* E \to E$ in $\mathcal{E}$ such that $P \argp{g} = f$ and, for each object $E''$ in $\mathcal{E}$, the natural map
\begin{align*}
\Hom[\mathcal{E}]{E''}{f^* E} 
& \to \set{ \tuple{h, f'} \in \Hom[\mathcal{E}]{E''}{E} \times \Hom[\mathcal{B}]{P \argp{E''}}{B'} }{ P \argp{h} = f \circ f' } \\
g' & \mapsto \tuple{g \circ g', P \argp{g'}}
\end{align*}
is a bijection.

Dually, a \strong{Grothendieck opfibration} is a functor $P : \mathcal{E} \to \mathcal{B}$ such that $\op{P} : \op{\mathcal{E}} \to \op{\mathcal{B}}$ is a Grothendieck fibration. %
%A \strong{Grothendieck opfibration} is a functor $P : \mathcal{E} \to \mathcal{B}$ such that, for every object $E$ in $\mathcal{E}$ and every morphism $f : P \argp{E} \to B'$ in $\mathcal{B}$, there exist an object $f_* E$ in $\mathcal{E}$ where $P \argp{f_* E} = B'$ and a bijection
%\[
%\Hom[\mathcal{E}]{f_* E}{E''} \cong \set{ \tuple{g, f'} \in \Hom[\mathcal{E}]{E}{E'} \times \Hom[\mathcal{B}]{B'}{P \argp{E''}} }{ P \argp{g} = f' \circ f }
%\]
%that is natural in $E'$.
\end{dfn}

\begin{remark}
Let $P : \mathcal{E} \to \mathcal{B}$ be a Grothendieck fibration and let $B$ be an object in $\mathcal{B}$. For any morphism $f : B' \to B$ in $\mathcal{B}$, the assignment $E \mapsto f^* E$ extends to a functor $\inv{P} \set{ B } \to \inv{P} \set{ B' }$; moreover, given a morphism $f' : B'' \to B'$ in $\mathcal{B}$, there is a canonical  isomorphism $f^{\prime *} f^* \cong \parens{f \circ f'}^*$ of functors $\inv{P} \set{ B } \to \inv{P} \set{ B'' }$. (More precisely, the assignment $f \mapsto f^*$ defines a  \emph{pseudofunctor}.) We say a Grothendieck fibration is \strong{split} if it is possible to choose the functors $f^*$ so that the assignment $f \mapsto f^*$ defines a functor $\op{\mathcal{B}} \to \cat{\Cat}$.
\end{remark}

\begin{example}
Let $\mathcal{C}$ be a category and let $\Func{\bracket{1}}{\mathcal{C}}$ be the category whose objects are the morphisms in $\mathcal{C}$ and whose morphisms are the commutative squares. Then the projection $\dom : \Func{\bracket{1}}{\mathcal{C}} \to \mathcal{C}$ is always a split Grothendieck fibration: the fiber over an object $A$ in $\mathcal{C}$ is the coslice category $\undercat{A}{\mathcal{C}}$, and for each morphism $f : A' \to A$ in $\mathcal{C}$, we can take the functor $f^* : \undercat{A}{\mathcal{C}} \to \undercat{A'}{\mathcal{C}}$ to be the one defined by precomposition.
\end{example}

\begin{thm}
\label{thm:Quillen-B}
Consider a pullback square in $\cat{\Cat}$:
\[
\begin{tikzcd}
\mathcal{E} \dar[swap]{p} \rar{v} &
\mathcal{F} \dar{q} \\
\mathcal{A} \rar[swap]{u} &
\mathcal{B}
\end{tikzcd}
\]
\begin{itemize}
\item If $q : \mathcal{F} \to \mathcal{B}$ is a Grothendieck opfibration such that the induced morphism $\inv{q} \set{b'} \to \inv{q} \set{b}$ is a weak homotopy equivalence of categories for every morphism $b' \to b$ in $\mathcal{B}$, then the above is a homotopy pullback diagram.

\item If $q : \mathcal{F} \to \mathcal{B}$ is a Grothendieck fibration such that the induced morphism $\inv{q} \set{b'} \to \inv{q} \set{b}$ is a weak homotopy equivalence of categories for every morphism $b \to b'$ in $\mathcal{B}$, then the above is a homotopy pullback diagram.
\end{itemize}
\end{thm}
\begin{proof}
The two claims are formally dual; we will prove the first version.

We may construct a commutative diagram in $\cat{\SSet}$ of the form below,
\[
\begin{tikzcd}
\nv{\mathcal{A}} \dar[swap]{i_A} \rar{\nv{u}} &
\nv{\mathcal{B}} \dar{i_B} &
\nv{\mathcal{F}} \dar{i_F} \lar[swap]{\nv{q}} \\
\hat{A} \rar[swap]{\hat{u}} &
\hat{B} &
\hat{F} \lar{\hat{q}}
\end{tikzcd}
\]
where the vertical arrows are weak homotopy equivalences, the horizontal arrows in the bottom row are Kan fibrations, and the objects in the bottom row are Kan complexes. Then, form the following pullback diagram in $\cat{\SSet}$:
\[
\begin{tikzcd}
\hat{E} \dar[swap]{\hat{p}} \rar{\hat{v}} & 
\hat{F} \dar{\hat{q}} \\
\hat{A} \rar[swap]{\hat{u}} &
\hat{B}
\end{tikzcd}
\]
Since $\hat{u} : \hat{A} \to \hat{B}$ and $\hat{q} : \hat{F} \to \hat{B}$ are Kan fibrations, the above is a homotopy pullback diagram. We wish to show that the induced morphism $i_E : \nv{\mathcal{E}} \to \hat{E}$ is a weak homotopy equivalence. 

For each object $b$ in $\mathcal{B}$, let $\commacat{q}{b} = \overcat{\mathcal{B}}{b} \times_\mathcal{B} \mathcal{E}$ be the comma category. Since $q : \mathcal{F} \to \mathcal{B}$ is a Grothendieck opfibration, the evident inclusion $\inv{q} \set{b} \to \commacat{q}{b}$ has a left adjoint and hence is a weak homotopy equivalence of categories (by \autoref{lem:homotopy.equivalences.of.categories}). We then have the following commutative diagram,
\[
\begin{tikzcd}
\inv{q} \set{ b } \dar \rar[hookrightarrow] &
\commacat{q}{b} \dar \\
\inv{q} \set{ b' } \rar[hookrightarrow] &
\commacat{q}{b'}
\end{tikzcd}
\]
where the horizontal arrows in the above diagram are weak homotopy equivalences, so our hypothesis on the functor $\inv{q} \set{ b } \to \inv{q} \set{ b' }$ implies that $\commacat{q}{b} \to \commacat{q}{b'}$ is also a weak homotopy equivalence. Thus, we may apply Quillen's Theorem B in its usual form\footnote{See \citep[\Sect 1]{Quillen:1973a} or \citep[\Chap III, \Sect 5.2]{GJ}.} to deduce that the pullback diagrams
\[
\mmhfill
\begin{tikzcd}
\inv{p} \set{a} \dar \rar[hookrightarrow] &
\mathcal{E} \dar{p} \\
\set{a} \rar[hookrightarrow] &
\mathcal{A}
\end{tikzcd}
\mmhfill
\begin{tikzcd}
\inv{q} \set{u \argp{a}} \dar \rar[hookrightarrow] &
\mathcal{F} \dar{q} \\
\set{u \argp{a}} \rar[hookrightarrow] &
\mathcal{B}
\end{tikzcd}
\mmhfill
\]
are homotopy pullback diagrams. Hence $i_E : \nv{\mathcal{E}} \to \hat{E}$ is a homotopy-fiberwise weak homotopy equivalence of objects over $\hat{A}$, and it follows that $i_E : \nv{\mathcal{E}} \to \hat{E}$ is a weak homotopy equivalence.
\end{proof}

\begin{cor}
\label{cor:binary.products.are.homotopy.products}
Let $\mathcal{A}$ and $\mathcal{B}$ be categories. Then the commutative diagram
\[
\begin{tikzcd}
\mathcal{A} \times \mathcal{B} \dar \rar &
\mathcal{B} \dar \\
\mathcal{A} \rar &
\bracket{0}
\end{tikzcd}
\]
where the arrows are the evident projections and $\bracket{0}$ is the terminal category,  is a homotopy pullback diagram.
\end{cor}
\begin{proof}
The unique functor $\mathcal{B} \to \bracket{0}$ is a Grothendieck fibration, so we may apply \autoref{thm:Quillen-B}.
\end{proof}

Let us recall Definition 9.1 from \citep{Dwyer-Kan:1980b}:

\begin{dfn}
Let $\mathcal{C}$ be a category and let $\mathcal{F} : \op{\mathcal{C}} \to \cat{\Cat}$ and $\mathcal{G} : \mathcal{C} \to \cat{\Cat}$ be functors. The \strong{two-sided  Grothendieck construction} $\mathcal{F} \otimes_\mathcal{C} \mathcal{G}$ is the following category:
\begin{itemize}
\item The objects are triples $\tuple{X, C, Y}$, where $C$ is an object in $\mathcal{C}$, $X$ is an object in $\mathcal{F} \argp{C}$, and $Y$ is an object in $\mathcal{G} \argp{C}$.

\item The morphisms $\tuple{X', C', Y'} \to \tuple{X, C, Y}$ are triples $\tuple{f, c, g}$, where $c : C' \to C$ is a morphism in $\mathcal{C}$, $f : X' \to \mathcal{F} \argp{c} \argp{X}$ is a morphism in $\mathcal{F} \argp{C'}$, and $g : \mathcal{G} \argp{c} \argp{Y'} \to Y$ is a morphism in $\mathcal{G} \argp{C}$; \confer the diagram below:
\[
\begin{tikzcd}
F' \rar{f} &
\mathcal{F} \argp{c} \argp{F} &
C' \dar{c} &
G' \dar[mapsto, swap]{\mathcal{G} \argp{c}} \\
&
F \uar[mapsto, swap]{\mathcal{F} \argp{c}} &
C &
\mathcal{G} \argp{c} \argp{G'} \rar[swap]{g} &
G
\end{tikzcd}
\] 

\item Composition and identities are defined in the obvious way.
\end{itemize}
\end{dfn}

\begin{remark}
\label{rem:Grothendieck.constructions.and.Grothendieck.fibrations}
Let $\mathcal{C}$ be a category and let $*$ be the constant functor $\mathcal{C} \to \cat{\Cat}$ with value $\mathbf{1}$ (the terminal category). Then, for any functor $\mathcal{F} : \op{\mathcal{C}} \to \cat{\Cat}$, the evident projection $\mathcal{F} \otimes_\mathcal{C} {*} \to \mathcal{C}$ is a split Grothendieck fibration.
\end{remark}

The two-sided Grothendieck construction $\mathcal{F} \otimes_\mathcal{C} \mathcal{G}$ can be thought of as being the homotopy colimit of $\mathcal{G}$ weighted by $\mathcal{F}$. Indeed, we have the following homotopy-invariance property:

\begin{lem}
\label{lem:two-sided.Grothendieck.construction.is.homotopical}
Let $\mathcal{C}$ be a category, let $\mathcal{F}, \mathcal{F}' : \op{\mathcal{C}} \to \cat{\Cat}$ and $\mathcal{G}, \mathcal{G}' : \mathcal{C} \to \cat{\Cat}$ be functors, and let $\phi : \mathcal{F}' \hoto \mathcal{F}$ and $\psi : \mathcal{G}' \hoto \mathcal{G}$ be natural transformations. If each $\phi_C : \mathcal{F}' \argp{C} \to \mathcal{F} \argp{C}$ and each $\psi_C : \mathcal{G}' \argp{C} \to \mathcal{G} \argp{C}$ is a weak homotopy equivalence of categories, then the induced functor $\phi \otimes_\mathcal{C} \psi : \mathcal{F}' \otimes_\mathcal{C} \mathcal{G}' \to \mathcal{F} \otimes_\mathcal{C} \mathcal{G}$ is also a weak homotopy equivalence of categories.
\end{lem}
\begin{proof} \openproof
This is Corollary 9.6 in \citep{Dwyer-Kan:1980b}.
\end{proof}

\section{Zigzags in relative categories}

We begin this section by introducing the main objects of study.  Recall the following pair of definitions from \citet{Barwick-Kan:2012a}:

\begin{dfn}
\ \noprelistbreak
\begin{itemize}
\item A \strong{relative category} is a pair $\mathcal{C} = \tuple{\und \mathcal{C}, \weq \mathcal{C}}$ where $\und \mathcal{C}$ is a category and $\weq \mathcal{C}$ is a (generally non-full) subcategory of $\und \mathcal{C}$ containing all the objects.

\item Given a relative category $\mathcal{C}$, a \strong{weak equivalence} in $\mathcal{C}$ is a morphism in $\weq \mathcal{C}$.

\item The \strong{homotopy category} of a relative category $\mathcal{C}$ is the category $\Ho \mathcal{C}$ obtained by freely inverting the weak equivalences in $\mathcal{C}$.

\item A relative category $\mathcal{C}$ is said to be \strong{saturated} when it satisfies the following condition: a morphism in $\mathcal{C}$ becomes invertible in $\Ho \mathcal{C}$ if and only if it is a weak equivalence in $\mathcal{C}$.
\end{itemize}
\end{dfn}

\begin{dfn}
Given relative categories $\mathcal{C}$ and $\mathcal{D}$:
\begin{itemize}
\item A \strong{relative functor} $\mathcal{C} \to \mathcal{D}$ is a functor $\und \mathcal{C} \to \und \mathcal{D}$ that restricts to a functor $\weq \mathcal{C} \to \weq \mathcal{D}$.

\item The \strong{relative functor category} $\RelFun{\mathcal{C}}{\mathcal{D}}$ is the relative category whose underlying category is the full subcategory of the ordinary functor category $\Func{\und \mathcal{C}}{\und \mathcal{D}}$ spanned by the relative functors, with the weak equivalences being the natural transformations whose components are weak equivalences in $\mathcal{D}$.
\end{itemize}
\end{dfn}
%
%\begin{remark}
%The 2-category of (small) categories admits several 2-fully faithful embeddings into the 2-category of (small) relative categories; unless otherwise stated, we will regard an ordinary category as \strong{minimal relative category} where the only weak equivalences are the identity morphisms. In particular, given an ordinary category $\mathcal{C}$ and a relative category $\mathcal{D}$, we will often tacitly identify the ordinary functor category $\Func{\mathcal{C}}{\mathcal{D}}$ with the relative functor category $\RelFun{\mathcal{C}}{\mathcal{D}}$.
%\end{remark}

Given a relative category $\mathcal{C}$, we are interested in understanding the morphisms in its homotopy category $\Ho \mathcal{C}$ in terms of the morphisms in $\mathcal{C}$ itself.  This immediately leads us to the following notions.

\begin{dfn}
\ \noprelistbreak
\begin{itemize}
\item A \strong{zigzag type} is a finite sequence of non-zero integers $\tuple{k_0, \ldots, k_n}$, where $n \ge 0$, such that for $0 \le i < n$, the sign of $k_i$ is the opposite of the sign of $k_{i+1}$.

\item Given a finite sequence of integers $\vec{k} = \tuple{k_0, \ldots, k_n}$, $\bracket{\vec{k}} = \bracket{k_0 ; \ldots ; k_n}$ is the relative category whose underlying category is freely generated by the graph
\[
\begin{tikzcd}
0 \rar[-] &
\cdots \rar[-] &
\abs{\vec{k}}
\end{tikzcd}
\]
where $\abs{\vec{k}} = \sum_{i=0}^{n} \abs{k_i}$ and (counting from the left) the first $\abs{k_0}$ arrows point rightward (\resp leftward) if $k_0 > 0$ (\resp $k_0 < 0$), the next $\abs{k_1}$ arrows point rightward (\resp leftward) if $k_1 > 0$ (\resp $k_1 < 0$), \etc, with the weak equivalences being generated by the leftward-pointing arrows.

\item A \strong{zigzag} in a relative category $\mathcal{C}$ of type $\bracket{\vec{k}}$ is a relative functor $\bracket{\vec{k}} \to \mathcal{C}$; given a zigzag, its \strong{domain} is the image of the object $0$ and its \strong{codomain} is the image of the object $\abs{\vec{k}}$.
\end{itemize}
\end{dfn}

\begin{example}
For example, $\bracket{-1 ; 2}$ denotes the relative category generated by the following graph,
\[
\begin{tikzcd}
0 \rar[leftarrow]{\simeq} &
1 \rar &
2 \rar &
3
\end{tikzcd}
\]
with $1 \to 0$ being the unique non-trivial weak equivalence.
\end{example}

\begin{remark}
For any $\bracket{k_0 ; \ldots ; k_n}$, there is a unique zigzag type $\tuple{l_0, \ldots, l_m}$ such that $\bracket{k_0 ; \ldots ; k_n} = \bracket{l_0 ; \ldots ; l_m}$. However, it is convenient to allow unnormalized notation, \eg, $\bracket{1 ; 1}$ instead of $\bracket{2}$, or $\bracket{0}$ instead of $\bracket{\thinspace}$.
\end{remark}

Any morphism in $\Ho \mathcal{C}$ is represented by a zigzag in $\mathcal{C}$, and hence one can describe the hom-sets in $\Ho \mathcal{C}$ as quotients of various sets of zigzags in $\mathcal{C}$ by the appropriate equivalence relations.  However, there is a more homotopically sensitive construction we can perform, where we instead obtain a \emph{category} of zigzags between two given objects of $\mathcal{C}$; we we will think of this as a \emph{space} of morphisms, following the philosophy laid out in \Sect 1.

\begin{dfn}
Let $X$ and $Y$ be objects in a relative category $\mathcal{C}$ and let $\vec{k}$ be a finite sequence of integers. The \strong{category of zigzags} in $\mathcal{C}$ from $X$ to $Y$ of type $\bracket{\vec{k}}$ is the category $\mathcal{C}^{\bracket{\vec{k}}} \argp{X, Y}$ defined below:
\begin{itemize}
\item The objects are the zigzags in $\mathcal{C}$ of type $\bracket{\vec{k}}$  whose domain is $X$ and whose codomain is $Y$.

\item The morphisms are commutative diagrams in $\mathcal{C}$ of the form
\[
\begin{tikzcd}
X \dar[equals] \rar[-] &
\bullet \dar[swap]{\simeq} \rar[-] &
\cdots \rar[-] &
\bullet \dar{\simeq} \rar[-] &
Y \dar[equals] \\
X \rar[-] &
\bullet \rar[-] &
\cdots \rar[-] &
\bullet \rar[-] &
Y
\end{tikzcd}
\]
where the top row is the domain, the bottom row is the codomain, and the vertical arrows are weak equivalences in $\mathcal{C}$.

\item Composition and identities are inherited from $\mathcal{C}$.
\end{itemize}

In other words, the objects (\resp morphisms) in $\mathcal{C}^{\bracket{\vec{k}}} \argp{X, Y}$ are certain hammocks of width $0$ (\resp $1$) in the sense of \citet{Dwyer-Kan:1980b}.
\end{dfn}

\begin{remark}
\label{rem:fibred.category.of.zigzags}
The following diagram is a pullback square,
\[
\begin{tikzcd}
\mathcal{C}^{\bracket{\vec{k}}} \argp{X, Y} \dar \rar[hookrightarrow] &
\weq \RelFun{\bracket{\vec{k}}}{\mathcal{C}} \dar{\prodtuple{\dom, \codom}} \\
\bracket{0} \rar[swap]{\tuple{X, Y}} &
\weq \mathcal{C} \times \weq \mathcal{C}
\end{tikzcd}
\]
where the top horizontal arrow is the evident inclusion and the bottom horizontal arrow is the functor $\bracket{0} \to \weq \mathcal{C} \times \weq \mathcal{C}$ corresponding to the object $\tuple{X, Y}$.
\end{remark}

In order to prove the main result, we will need to collect some assorted facts about these categories of zigzags, which will occupy the remainder of this section.  

\begin{prop}
\label{prop:Grothendieck.bifibration.of.zigzags}
\needspace{3\baselineskip}
Let $\mathcal{C}$ be a relative category, let $\vec{k} = \tuple{k_0, \ldots, k_n}$ be a zigzag type, and assume $k_0 < 0$ and $k_n < 0$.
\begin{enumerate}[(i)]
\item $\mathcal{C}^{\bracket{\vec{k}}} \argp{\blank, \blank}$ is (the object part of) a functor $\weq \mathcal{C} \times \weq \op{\mathcal{C}} \to \cat{\Cat}$ and there is a canonical isomorphism
\[
{*} \otimes_{\weq \mathcal{C}} \mathcal{C}^{\bracket{\vec{k}}} \otimes_{\weq \mathcal{C}} {*} \cong \weq \RelFun{\bracket{\vec{k}}}{\mathcal{C}}
\]
of categories over $\weq \mathcal{C} \times \weq \mathcal{C}$.

\item In particular, the domain projection $\dom : \weq \RelFun{\bracket{\vec{k}}}{\mathcal{C}} \to \weq \mathcal{C}$ is a split Grothendieck opfibration, and the codomain projection $\codom : \weq \RelFun{\bracket{\vec{k}}}{\mathcal{C}} \allowbreak \to \weq \mathcal{C}$ is a split Grothendieck fibration.
\end{enumerate}
\end{prop}
\begin{proof}
(i). This can be verified directly. 

\bigskip\noindent
(ii). Apply \autoref{rem:Grothendieck.constructions.and.Grothendieck.fibrations}.
\end{proof}

\begin{remark}
The observation above will be the backbone of \autoref{prop:classifying.space.of.morphisms.with.homotopical.three-arrow.calculi}: the point is that $\dom : \weq \RelFun{\bracket{-1 ; 1 ; -1}}{\mathcal{C}} \to \weq \mathcal{C}$ is a Grothendieck opfibration whose fiber over an object $X$ in $\mathcal{C}$ is a category that is itself equipped with a Grothendieck fibration to $\weq \mathcal{C}$ whose fiber over an object $Y$ in $\mathcal{C}$ is the zigzag category $\mathcal{C}^{\bracket{-1 ; 1 ; -1}} \argp{X, Y}$.
\end{remark}

\begin{lem}
\label{lem:composing.weak.equivalences.in.zigzags}
Let $\mathcal{C}$ be a relative category, let $X$ and $Y$ be objects in $\mathcal{C}$, and let $\tuple{k_0, \ldots, k_{i-1}}$ and $\tuple{k_{i+1}, \ldots, k_n}$ be finite sequences of integers (possibly of length zero). Then the two evident functors
\[
s_0, s_1 : \mathcal{C}^{\bracket{k_0 ; \ldots ; k_{i-1} ; -1 ; k_{i+1} ; \ldots ; k_n}} \argp{X ; Y} \to \mathcal{C}^{\bracket{k_0 ; \ldots ; k_{i-1} ; -2 ; k_{i+1} ; \ldots ; k_n}} \argp{X, Y}
\]
defined by inserting an identity morphism and the evident functor 
\[
d : \mathcal{C}^{\bracket{k_0 ; \ldots ; k_{i-1} ; -2 ; k_{i+1} ; \ldots ; k_n}} \argp{X, Y} \to \mathcal{C}^{\bracket{k_0 ; \ldots ; k_{i-1} ; -1 ; k_{i+1} ; \ldots ; k_n}} \argp{X, Y}
\]
defined by composing the two leftward-pointing arrows are weak homotopy equivalences of categories.
\end{lem}
\begin{proof}
Clearly, $d \circ s_0 = d \circ s_1 = \id$; on the other hand, the commutative diagrams
\[
\begin{tikzcd}
X \dar[equals] \rar[-, squiggly] &
\bullet \dar[equals] \rar[leftarrow]{v} &
\bullet \dar[swap]{v} \rar[leftarrow]{u} &
\bullet \dar[equals] \rar[-, squiggly] &
Y \dar[equals] \\
X \rar[-, squiggly] &
\bullet \rar[equals] &
\bullet \rar[swap, leftarrow]{v \circ u} &
\bullet \rar[-, squiggly] &
Y
\end{tikzcd}
\]
\[
\begin{tikzcd}
X \dar[equals] \rar[-, squiggly] &
\bullet \dar[equals] \rar[leftarrow]{v \circ u} &
\bullet \dar{u} \rar[equals] &
\bullet \dar[equals] \rar[-, squiggly] &
Y \dar[equals] \\
X \rar[-, squiggly] &
\bullet \rar[swap, leftarrow]{v} &
\bullet \rar[swap, leftarrow]{u} &
\bullet \rar[-, squiggly] &
Y
\end{tikzcd}
\]
define (respectively) natural transformations $\id \hoto s_0 \circ d$ and $s_1 \circ d \hoto \id$, so by \autoref{lem:homotopy.equivalences.of.categories}, all three functors are indeed weak homotopy equivalences of categories.
\end{proof}

\begin{lem}
\label{lem:comparing.the.arrow.category.and.the.three-arrow.zigzag.category}
Let $\mathcal{C}$ be a relative category and let $k$ be a natural number. Then the evident functor $s^2 : \weq \RelFun{\bracket{k}}{\mathcal{C}} \to \weq \RelFun{\bracket{-1 ; k ; -1}}{\mathcal{C}}$ defined by inserting (two) identity morphisms is a weak homotopy equivalence of categories. 
%making the following diagram commute:
%\[
%\begin{tikzcd}
%\weq \RelFun{\bracket{k}}{\mathcal{C}} \dar[swap]{\prodtuple{\dom, \codom}} \rar{s^2} &
%\weq \RelFun{\bracket{-1 ; k ; -1}}{\mathcal{C}} \dar{\prodtuple{\dom, \codom}} \\
%\weq \mathcal{C} \times \weq \mathcal{C} \rar[equals] &
%\weq \mathcal{C} \times \weq \mathcal{C}
%\end{tikzcd}
%\]
\end{lem}
\begin{proof}
%It is clear that the diagram in question commutes, and it remains to be shown that $s^2$ is a weak homotopy equivalence of categories. But 
For every zigzag in $\mathcal{C}$ of type $\bracket{-1 ; k ; -1}$, say
\[
\begin{tikzcd}
X \rar[leftarrow]{v} &
\tilde{X} \rar &
\cdots \rar &
\hat{Y} \rar[leftarrow]{u} &
Y
\end{tikzcd}
\]
there is a natural commutative diagram in $\mathcal{C}$ of the form below,
\[
\begin{tikzcd}
X \dar[swap, leftarrow]{v} \rar[leftarrow]{v} &
\tilde{X} \dar[equals] \rar &
\cdots \rar &
\hat{Y} \dar[equals] \rar[leftarrow]{u} &
Y \dar[equals] \\
\tilde{X} \dar[equals] \rar[equals] &
\tilde{X} \dar[equals] \rar &
\cdots \rar &
\hat{Y} \dar[equals] \rar[leftarrow]{u} &
Y \dar{u} \\
\tilde{X} \rar[equals] &
\tilde{X} \rar &
\cdots \rar &
\hat{Y} \rar[equals] &
\hat{Y}
\end{tikzcd}
\]
so the functor $r^2 : \weq \RelFun{\bracket{-1 ; k ; -1}}{\mathcal{C}} \to \weq \RelFun{\bracket{k}}{\mathcal{C}}$ defined by discarding the two outermost arrows satisfies $r^2 \circ s^2 = \id$, and there is a zigzag of natural transformations between $\id$ and $s^2 \circ r^2$. In particular, by \autoref{lem:homotopy.equivalences.of.categories}, $s^2$ is a weak homotopy equivalence of categories.
\end{proof}

\section{The main result}

In this section we state and prove our main result\hairspace ---\hairspace namely that a saturated relative category which enjoys a certain factorization condition will have the property that its Rezk classification diagram is a complete Segal space up to Reedy-fibrant replacement.  We begin by recalling this factorization condition, which is a variation on the ``homotopy calculus of fractions'' introduced in \citep{Dwyer-Kan:1980b}.

\begin{dfn}
\label{dfn:homotopical.three-arrow.calculus}
\needspace{3\baselineskip}
A relative category $\mathcal{C}$ admits a \strong{homotopical three-arrow calculus} if it satisfies the following condition:
\begin{itemize}
\item For all natural numbers $k$ and $l$ and all objects $X$ and $Y$ in $\mathcal{C}$, the evident functor
\[
\mathcal{C}^{\bracket{-1 ; k ; l ; -1}} \argp{X, Y} \to \mathcal{C}^{\bracket{-1 ; k ; -1 ; l ; -1}} \argp{X, Y}
\]
defined by inserting an identity morphism is a weak homotopy equivalence of categories.
\end{itemize}
\end{dfn}

\begin{remark}
\label{rem:2-out-of-3.and.the.homotopical.three-arrow.calculus}
Let $\mathcal{C}$ be a relative category and let $\mathcal{W}$ be $\weq \mathcal{C}$ considered as a relative category where all morphisms are weak equivalences. Then (recalling \autoref{lem:composing.weak.equivalences.in.zigzags}) the following are equivalent:
\begin{enumerate}[(i)]
\item $\mathcal{C}$ admits a homotopy calculus of fractions in the sense of \citet{Dwyer-Kan:1980b}.

\item Both $\mathcal{C}$ and $\mathcal{W}$ admit a homotopical three-arrow calculus in the sense of \autoref{dfn:homotopical.three-arrow.calculus}.
\end{enumerate}
Moreover, if the weak equivalences in $\mathcal{C}$ have the 2-out-of-3 property, then $\mathcal{W}$ admits a homotopical three-arrow calculus if $\mathcal{C}$ does.

However, $\mathcal{C}$ may admit a homotopical three-arrow calculus even when $\mathcal{W}$ does not have the 2-out-of-3 property. For example, consider the relative category $\mathcal{C}$ whose underlying category is $\bracket{2}$ and whose weak equivalences are generated by the unique morphisms $0 \to 2$ and $1 \to 2$. Clearly, weak equivalences in $\mathcal{C}$ do not have the 2-out-of-3 property. On the other hand, for any $X$ and $Y$ in $\mathcal{C}$ and any finite sequence $\vec{k}$ (possibly of length zero), the category $\mathcal{C}^{\bracket{-1 ; \vec{k} ; -1}} \argp{X, Y}$ is a poset with a maximum element, so is contractible (by \autoref{lem:homotopy.equivalences.of.categories}). Thus, $\mathcal{C}$ indeed admits a homotopical three-arrow calculus.
\end{remark}

\begin{example}
Let $\mathcal{C}$ be a partial model category in the sense of \citet{Barwick-Kan:2011}. Then, by Proposition 8.2 in \citep{Dwyer-Kan:1980b} and \autoref{rem:2-out-of-3.and.the.homotopical.three-arrow.calculus}, $\mathcal{C}$ admits a homotopical three-arrow calculus. 
\end{example}

\begin{example}
\needspace{3\baselineskip}
Let $\mathcal{M}$ be a model category and let $\mathcal{C}$ be a full subcategory of $\mathcal{M}$. Suppose $\mathcal{C}$ is homotopically replete, \ie, satisfies the following condition:
\begin{itemize}
\item For any weak equivalence $w : X \to Y$ in $\mathcal{M}$, if either $X$ or $Y$ is in $\mathcal{C}$, then $X$, $Y$, and $w$ are all in $\mathcal{C}$.
\end{itemize}
Then $\mathcal{C}$ admits a homotopical three-arrow calculus. Essentially, one examines the argument of paragraph 8.1 in \citep{Dwyer-Kan:1980c} and notes it goes through with $\mathcal{C}$ in place of $\mathcal{M}$. (Observe that the condition on $\mathcal{C}$ ensures that it is closed in $\mathcal{M}$ under pullbacks along trivial fibrations, the construction of simplicial resolutions, \etc.) In particular, $\mathcal{M}$ admits a homotopical three-arrow calculus.
\end{example}

We should think of a homotopical three-arrow calculus as a guarantee that we can reduce our zigzags in $\mathcal{C}$ from longer to shorter: the exact condition says that up to a certain suitable notion of equivalence, we can remove the middle weak equivalence in a zigzag of type $\bracket{-1 ; k ; -1 ; l ; -1}$ to obtain a zigzag of type $\bracket{-1 ; k ; l ; -1}$, and moreover that this reduction is sensitive to the homotopical information contained in the categories of zigzags involved.  In fact, the presence of a homotopical three-arrow calculus allows us to reduce \emph{all} zigzags in $\mathcal{C}$ in this homotopically sensitive way to the smallest sort that we might hope. (Recall the discussion in the introduction!) More precisely, there is the following theorem:

\begin{thm}[Dwyer and Kan]
\label{thm:fundemental.theorem.of.homotopical.three-arrow.calculi}
\needspace{3\baselineskip}
Let $\mathcal{C}$ be a relative category and let $\ul{\LH \mathcal{C}}$ be the hammock localization. 
\begin{enumerate}[(i)]
\item If $\mathcal{C}$ admits a homotopical three-arrow calculus, then the reduction map
\[
\nv{\mathcal{C}^{\bracket{-1 ; 1 ; -1}} \argp{X, Y}} \to \ulHom[\LH \mathcal{C}]{X}{Y}
\]
is a weak homotopy equivalence of simplicial sets.

\item The reduction map $\nv{\mathcal{C}^{\bracket{-1 ; 1 ; -1}} \argp{X, Y}} \to \ulHom[\LH \mathcal{C}]{X}{Y}$ is natural in the following sense: given weak equivalences $X \to X'$ and $Y' \to Y$ in $\mathcal{C}$, the following diagram commutes in $\cat{\SSet}$,
\[
\begin{tikzcd}
\nv{\mathcal{C}^{\bracket{-1 ; 1 ; -1}} \argp{X, Y}} \dar \rar &
\ulHom[\LH \mathcal{C}]{X}{Y} \dar \\
\nv{\mathcal{C}^{\bracket{-1 ; 1 ; -1}} \argp{X', Y'}} \rar &
\ulHom[\LH \mathcal{C}]{X'}{Y'}
\end{tikzcd}
\]
where the vertical arrow on the left is defined by composition and the vertical arrow on the right is defined by concatenation.
\end{enumerate}
\end{thm}
\begin{proof} \openproof
(i). This is Proposition 6.2 in \citep{Dwyer-Kan:1980b}. Note that the second half of the `homotopy calculus of fractions' condition is not used, so it does indeed suffice to have a homotopical three-arrow calculus.

\bigskip\noindent
(ii). Immediate, from the definitions.
\end{proof}

\begin{remark}
Statement (ii) above does not appear explicitly in \citep{Dwyer-Kan:1980b}, but we will need it later. We also note the following corrections to the proof of (i) given in \opcit:
\begin{itemize}
\item The functor $B : \mathbf{II} \to \mathbf{II}$ should instead be given by the following formula:
\[
\tuple{S, T} \mapsto \tuple{\set{s + 1}{s \in S}, \set{1} \cup \set{t + 1}{t \in T}}
\]
 
\item The formula given for $A$ does not define a functor on the whole of $\mathbf{II}$; instead, define $B A$ to be the functor given by the following formula:
\[
\tuple{S, T} \mapsto \tuple{\set{ 2, \ldots, \card{S} + 1 }, \set{ 1 }}
\]

\item In the last line, `5.1 (ii)' should be `6.1 (ii)'.
\end{itemize}
\end{remark}

\begin{cor}
\label{cor:action.of.weak.equivalences.on.three-arrow.zigzags}
Let $\mathcal{C}$ be a relative category. If $\mathcal{C}$ admits a homotopical three-arrow calculus, then for any weak equivalences $X \to X'$ and $Y' \to Y$ in $\mathcal{C}$, the induced functor
\[
\mathcal{C}^{\bracket{-1 ; 1 ; -1}} \argp{X, Y} \to \mathcal{C}^{\bracket{-1 ; 1 ; -1}} \argp{X', Y'}
\]
is a weak homotopy equivalence of categories.
\end{cor}
\begin{proof} \openproof
Use naturality (as in \autoref{thm:fundemental.theorem.of.homotopical.three-arrow.calculi}) and Proposition 3.3 in \citep{Dwyer-Kan:1980b}.
\end{proof}

We are now ready to compute (with assumptions) the homotopy fibers of the functor $\prodtuple{\dom, \codom} : \weq \Func{\bracket{1}}{\mathcal{C}} \to \mathcal{C} \times \mathcal{C}$, or, in other words, the hom-spaces of the Rezk classification diagram of $\mathcal{C}$.

\begin{prop}
\label{prop:classifying.space.of.morphisms.with.homotopical.three-arrow.calculi}
\needspace{2.5\baselineskip}
Let $\mathcal{C}$ be a relative category and let $\mathcal{W} = \weq \mathcal{C}$.
\begin{enumerate}[(i)]
\item There is a pullback diagram in $\cat{\Cat}$ of the form below,
\[
\begin{tikzcd}[column sep=7.5ex]
\weq \RelFun{\bracket{-1 ; 1 ; -1}}{\mathcal{C}} \dar \rar &
\weq \RelFun{\bracket{1}}{\mathcal{C}} \dar{\prodtuple{\dom, \codom}} \\
\Func{\bracket{1}}{\mathcal{W}} \times \Func{\bracket{1}}{\mathcal{W}} \rar[swap]{\codom \times \dom} &
\mathcal{W} \times \mathcal{W}
\end{tikzcd}
\]
and, moreover, the horizontal arrows in the diagram are weak homotopy equivalences of categories.

\item For each pair $\tuple{X, Y}$ of objects in $\mathcal{C}$, we have the following pullback diagram in $\cat{\Cat}$,
\[
\begin{tikzcd}
\mathcal{C}^{\bracket{-1 ; 1 ; -1}} \argp{X, Y} \dar \rar &
\weq \RelFun{\bracket{-1 ; 1 ; -1}}{\mathcal{C}} \dar{\prodtuple{\dom, \codom}} \\
\bracket{0} \rar[swap]{\tuple{X, Y}} &
\mathcal{W} \times \mathcal{W}
\end{tikzcd}
\]
where $\tuple{X, Y} : \bracket{0} \to \mathcal{W} \times \mathcal{W}$ is the functor corresponding to the object $\tuple{X, Y}$ in $\mathcal{W} \times \mathcal{W}$.

\item If $\mathcal{C}$ admits a homotopical three-arrow calculus, then we have a homotopy pullback diagram in $\cat{\Cat}$ of the form below,
\[
\begin{tikzcd}
\mathcal{C}^{\bracket{-1 ; 1 ; -1}} \argp{X, Y} \dar \rar &
\weq \RelFun{\bracket{1}}{\mathcal{C}} \dar{\prodtuple{\dom, \codom}} \\
\overcat{\mathcal{W}}{X} \times \undercat{Y}{\mathcal{W}} \rar &
\mathcal{W} \times \mathcal{W}
\end{tikzcd}
\]
where $\overcat{\mathcal{W}}{X}$ (\resp $\undercat{Y}{\mathcal{W}}$) is the slice (\resp coslice) category and the bottom horizontal arrow is defined by the evident projections.
\end{enumerate}
\end{prop}
\begin{proof}
(i). It is clear that we have a pullback diagram of the required form, the top horizontal arrow is a weak homotopy equivalence by \autoref{lem:comparing.the.arrow.category.and.the.three-arrow.zigzag.category}, and a similar argument shows that the bottom horizontal arrow is a weak homotopy equivalence as well.

\bigskip\noindent
(ii). This is \autoref{rem:fibred.category.of.zigzags} in the case $\bracket{\vec{k}} = \bracket{-1 ; 1 ; -1}$.

\bigskip\noindent
(iii). Consider the following commutative diagram in $\cat{\Cat}$,
\[
\begin{tikzcd}[column sep=7.5ex]
\mathcal{C}^{\bracket{-1 ; 1 ; -1}} \argp{X, Y} \dar \rar[hookrightarrow] 
\drar[start anchor=center, end anchor=center, draw=none][font=\normalsize, description]{\textup{(A)}} &
\weq \RelFun{\bracket{-1 ; 1 ; -1}}{\mathcal{C}} \dar \rar
\drar[start anchor=center, end anchor=center, draw=none][font=\normalsize, description]{\textup{(B)}} &
\weq \RelFun{\bracket{1}}{\mathcal{C}} \dar{\prodtuple{\dom, \codom}} \\
\overcat{\mathcal{W}}{X} \times \undercat{Y}{\mathcal{W}} \dar \rar[hookrightarrow] 
\drar[start anchor=center, end anchor=center, draw=none][font=\normalsize, description]{\textup{(C)}} &
\Func{\bracket{1}}{\mathcal{W}} \times \Func{\bracket{1}}{\mathcal{W}} \dar{\codom \times \dom} \rar[swap]{\dom \times \codom} &
\mathcal{W} \times \mathcal{W} \\
\bracket{0} \rar[swap]{\tuple{X, Y}} &
\mathcal{W} \times \mathcal{W}
\end{tikzcd}
\]
where every square is a pullback diagram. We wish to prove that rectangle (AB) is a homotopy pullback diagram, and since (B) is a homotopy pullback diagram, it suffices (by the homotopy pullback pasting lemma) to verify that (A) is a homotopy pullback diagram; however (by \autoref{lem:homotopy.equivalences.of.categories}) the vertical arrows in (C) are also weak homotopy equivalences, so (C) is a homotopy pullback diagram, and hence it is enough to show that the rectangle (AC) is a homotopy pullback diagram.

Let $\mathcal{H}_X : \op{\mathcal{W}} \to \cat{\Cat}$ be the diagram $\mathcal{C}^{\bracket{-1 ; 1 ; -1}} \argp{X, \blank}$. Then by \autoref{thm:Quillen-B}, \autoref{lem:two-sided.Grothendieck.construction.is.homotopical}, \autoref{prop:Grothendieck.bifibration.of.zigzags}, and \autoref{cor:action.of.weak.equivalences.on.three-arrow.zigzags}, the pullback diagrams
\[
\mmhfill
\begin{tikzcd}
\mathcal{H}_X \otimes_\mathcal{W} {*} \dar \rar &
\weq \RelFun{\bracket{-1 ; 1 ; -1}}{\mathcal{C}} \dar{\dom} \\
\bracket{0} \rar[swap]{X} &
\mathcal{W}
\end{tikzcd}
\mmhfill
\begin{tikzcd}
\mathcal{C}^{\bracket{-1 ; 1 ; -1}} \argp{X, Y} \dar \rar &
\mathcal{H}_X \otimes_\mathcal{W} {*} \dar \\
\bracket{0} \rar[swap]{Y} &
\mathcal{W}
\end{tikzcd}
\mmhfill
\]
are homotopy pullback diagrams. Thus, in the diagram shown below,
\[
\begin{tikzcd}
\mathcal{C}^{\bracket{-1 ; 1 ; -1}} \argp{X, Y} \dar \rar 
\drar[start anchor=center, end anchor=center, draw=none][font=\normalsize, description]{\textup{(D)}} &
\mathcal{H}_X \otimes_\mathcal{W} {*} \dar \rar 
\drar[start anchor=center, end anchor=center, draw=none][font=\normalsize, description]{\textup{(E)}} &
\weq \RelFun{\bracket{-1 ; 1 ; -1}}{\mathcal{C}} \dar{\prodtuple{\dom, \codom}} \\
\bracket{0} \rar[swap]{Y} &
\mathcal{W} \dar \rar[swap]{X \times \id_\mathcal{W}} 
\drar[start anchor=center, end anchor=center, draw=none][font=\normalsize, description]{\textup{(F)}} &
\mathcal{W} \times \mathcal{W} \dar{\mathrm{pr}_1} \\
&
\bracket{0} \rar[swap]{X} &
\mathcal{W}
\end{tikzcd}
\]
we know that (D) and (EF) are homotopy pullback diagrams; however \autoref{cor:binary.products.are.homotopy.products} says that the evident pullback diagram
\[
\begin{tikzcd}
\mathcal{W} \times \mathcal{W} \dar[swap]{\mathrm{pr}_1} \rar{\mathrm{pr}_2} &
\mathcal{W} \dar \\
\mathcal{W} \rar &
\bracket{0}
\end{tikzcd}
\]
is a homotopy pullback diagram, so (F) and (E) are also homotopy pullback diagrams. In particular,
\[
\begin{tikzcd}
\mathcal{C}^{\bracket{-1 ; 1 ; -1}} \argp{X, Y} \dar \rar[hookrightarrow] &
\weq \RelFun{\bracket{-1 ; 1 ; -1}}{\mathcal{C}} \dar{\prodtuple{\dom, \codom}} \\
\bracket{0} \rar[swap]{\tuple{X, Y}} &
\mathcal{W} \times \mathcal{W}
\end{tikzcd}
\]
is a homotopy pullback diagram, as required.
\end{proof}

We will also need a technical result concerning categories of longer zigzags in $\mathcal{C}$ and their behavior with respect to replacing the domain and codomain along weak equivalences.

\begin{lem}
\label{lem:action.of.weak.equivalences.on.extended.three-arrow.zigzags}
Let $\mathcal{C}$ be a relative category and let $n$ be a positive integer. If $\mathcal{C}$ admits a homotopical three-arrow calculus, then for any weak equivalences $X \to X'$ and $Y' \to Y$ in $\mathcal{C}$, the induced functor
\[
\mathcal{C}^{\bracket{-1 ; n ; -1}} \argp{X, Y} \to \mathcal{C}^{\bracket{-1 ; n ; -1}} \argp{X', Y'}
\]
is a weak homotopy equivalence of categories.
\end{lem}
\begin{proof} 
Since the class of weak homotopy equivalences of categories is closed under composition, it suffices to verify the claim when either $X \to X'$ or $Y' \to Y$ is an identity morphism; but the two cases are formally dual, so it is enough to prove the claim in the first case.

The special case $n = 1$ is \autoref{cor:action.of.weak.equivalences.on.three-arrow.zigzags}. In general, we have the following commutative diagram,
\[
\begin{tikzcd}
\mathcal{C}^{\bracket{-1 ; n + 1; -1}} \argp{X, Y} \dar \rar &
\mathcal{C}^{\bracket{-1 ; 1 ; -2 ; n ; -1}} \argp{X, Y} \dar \\
\mathcal{C}^{\bracket{-1 ; n + 1 ; -1}} \argp{X, Y'} \rar &
\mathcal{C}^{\bracket{-1 ; 1 ; -2 ; n ; -1}} \argp{X, Y'}
\end{tikzcd}
\]
where the horizontal arrows are the evident functors defined by inserting (two) identity morphisms. The horizontal arrows are weak homotopy equivalences of categories by \autoref{lem:composing.weak.equivalences.in.zigzags} and the hypothesis that $\mathcal{C}$ admits a homotopical three-arrow calculus, and the right vertical arrow is a weak homotopy equivalence of categories by Proposition 9.4 in \citep{Dwyer-Kan:1980b} and \autoref{lem:two-sided.Grothendieck.construction.is.homotopical}; thus, by the 2-out-of-3 property, the left vertical arrow is also a weak homotopy equivalence of categories, as required.
\end{proof}

We now prove that the Rezk classification diagram of a relative category admitting a homotopical three-arrow calculus has the asserted Reedy homotopy type, \ie, that any Reedy-fibrant replacement satisfies the Segal condition. This result sits at the heart of our main theorem (\ref{thm:homotopical.three-arrow.calculi.and.the.Rezk.classification.diagram}).

\begin{prop}
\label{prop:homotopical.three-arrow.calculi.and.the.Segal.condition}
Let $\mathcal{C}$ be a relative category and let $n$ be a positive integer.
If $\mathcal{C}$ admits a homotopical three-arrow calculus, then we have a homotopy pullback diagram in $\cat{\Cat}$ of the form below,
\[
\begin{tikzcd}
\weq \RelFun{\bracket{n+1}}{\mathcal{C}} \dar[swap]{p} \rar{d_0} &
\weq \RelFun{\bracket{n}}{\mathcal{C}} \dar{\dom} \\
\weq \RelFun{\bracket{1}}{\mathcal{C}} \rar[swap]{\codom} &
\weq \mathcal{C}
\end{tikzcd}
\]
where $p : \weq \RelFun{\bracket{n+1}}{\mathcal{C}} \to \weq \RelFun{\bracket{1}}{\mathcal{C}}$ is the functor defined by sending each composable sequence of morphisms in $\mathcal{C}$ of length $n + 1$, say
\[
\begin{tikzcd}
X_0 \rar{f_1} &
X_1 \rar &
\cdots \rar &
X_{n+1}
\end{tikzcd}
\]
(considered as an object in $\weq \Func{\bracket{n+1}}{\mathcal{C}}$), to the morphism $f_1 : X_0 \to X_1$ (considered as an object in $\weq \Func{\bracket{1}}{\mathcal{C}}$).
\end{prop}
\begin{proof}
Let $T_1 = \bracket{-1 ; 1 ; -1}$, $T_n = \bracket{-1 ; n ; -1}$, and $M_n = \bracket{-1; 1 ; -2 ; n ; -1}$. We have a commutative cube in $\cat{\Cat}$ of the form below,
\[
\begin{tikzcd}[column sep=3.0ex, row sep=3.0ex]
\weq \RelFun{\bracket{n+1}}{\mathcal{C}} \arrow{dd}[swap]{p} \drar \arrow{rr}{d_0} &&
\weq \RelFun{\bracket{n}}{\mathcal{C}} \arrow{dd}[near end]{\dom} \drar \\
&
\weq \RelFun{M_n}{\mathcal{C}} \arrow[crossing over]{rr} &&
\weq \RelFun{T_n}{\mathcal{C}} \arrow{dd}{\dom} \\
\weq \RelFun{\bracket{1}}{\mathcal{C}} \drar \arrow{rr}[swap, near end]{\codom} &&
\weq \mathcal{C} \drar[equals] \\
&
\weq \RelFun{T_1}{\mathcal{C}} \arrow[leftarrow, crossing over]{uu} \arrow[swap]{rr}{\codom} &&
\weq \mathcal{C}
\end{tikzcd}
\]
where the non-trivial oblique arrows are defined by inserting identity morphisms and both the front and back faces of the cube are pullback squares in $\cat{\Cat}$; moreover, by \autoref{thm:Quillen-B}, \autoref{lem:two-sided.Grothendieck.construction.is.homotopical},  \autoref{prop:Grothendieck.bifibration.of.zigzags}, and \autoref{lem:action.of.weak.equivalences.on.extended.three-arrow.zigzags}, the front face is a homotopy pullback diagram. Since the diagonal arrows are weak homotopy equivalences (by lemmas~\ref{lem:composing.weak.equivalences.in.zigzags} and~\ref{lem:comparing.the.arrow.category.and.the.three-arrow.zigzag.category} plus the hypothesis that $\mathcal{C}$ admits a homotopical three-arrow calculus), it follows that the back face is also a homotopy pullback diagram.
\end{proof}

We now come to the main theorem.

\begin{thm}
\label{thm:homotopical.three-arrow.calculi.and.the.Rezk.classification.diagram}
Let $\mathcal{C}$ be a relative category and let $\Nv{\mathcal{C}}_{\bullet}$ be the Rezk classification diagram for $\mathcal{C}$, \ie, the bisimplicial set defined by the following formula,
\[
\Nv{\mathcal{C}}_n = \nv{\weq \RelFun{\bracket{n}}{\mathcal{C}}}
\]
and let $\widehat{\Nv{\mathcal{C}}}_{\bullet}$ be any Reedy-fibrant replacement for $\Nv{\mathcal{C}}_{\bullet}$. If $\mathcal{C}$ admits a homotopical three-arrow calculus, then:
\begin{enumerate}[(i)]
\item $\widehat{\Nv{\mathcal{C}}}_{\bullet}$ is a Segal space.

\item Assuming weak equivalences in $\mathcal{C}$ have the 2-out-of-3 property, $\widehat{\Nv{\mathcal{C}}}_{\bullet}$ is a complete Segal space if and only if $\mathcal{C}$ is a saturated relative category.
\end{enumerate}
\end{thm}
\begin{proof}
(i). Since $\widehat{\Nv{\mathcal{C}}}_{\bullet}$ is a Reedy-fibrant bisimplicial set, it suffices to verify that the following diagram is a homotopy pullback diagram for every positive integer $n$,
\[
\begin{tikzcd}
\widehat{\Nv{\mathcal{C}}}_{n+1} \dar \rar{d_0} &
\widehat{\Nv{\mathcal{C}}}_n \dar \\
\widehat{\Nv{\mathcal{C}}}_1 \rar[swap]{d_0} &
\widehat{\Nv{\mathcal{C}}}_0
\end{tikzcd}
\]
where the right (\resp left) vertical arrow is the iterated face operator $d_1 \circ \cdots \circ d_n$ (\resp $d_2 \circ \cdots \circ d_{n+1}$); however we have a Reedy weak equivalence $\Nv{\mathcal{C}}_{\bullet} \to \widehat{\Nv{\mathcal{C}}}_{\bullet}$, so the claim is a consequence of \autoref{prop:homotopical.three-arrow.calculi.and.the.Segal.condition}.

\bigskip\noindent
(ii). For the `if' direction, we repeat the argument of the last paragraph of the proof of Theorem 8.3 in \citep{Rezk:2001}: since $\mathcal{C}$ is saturated, a vertex of $\Nv{\mathcal{C}}_1$ is an equivalence (\ie, becomes an isomorphism in $\Ho \mathcal{C}$) if and only if it is a weak equivalence in $\mathcal{C}$, so the space of homotopy equivalences in $\Nv{\mathcal{C}}_{\bullet}$ is the nerve $\nv{\Func{\bracket{1}}{\weq \mathcal{C}}}$. Following \autoref{lem:homotopy.equivalences.of.categories}, the induced morphism $\Nv{\mathcal{C}}_0 \to \nv{\Func{\bracket{1}}{\weq \mathcal{C}}}$ is a homotopy equivalence of simplicial sets, so $\Nv{\mathcal{C}}_{\bullet}$ is indeed a complete Segal space if $\mathcal{C}$ is saturated.

For the `only if' direction, we simply observe that if $f : X \to Y$ is a morphism in $\mathcal{C}$ that becomes an isomorphism in $\Ho \mathcal{C}$, then the corresponding vertex of $\Nv{\mathcal{C}}_1$ must be in the same connected component as the vertex corresponding to $\id_X$ (because $\Nv{\mathcal{C}}_{\bullet}$ is a complete Segal space). Thus, by using the 2-out-of-3 property of weak equivalences and induction, we deduce that $f$ is a weak equivalence in $\mathcal{C}$ if $f$ becomes an isomorphism in $\Ho \mathcal{C}$.
\end{proof}

\begin{cor}
Let $\mathcal{M}$ be a model category, let $\Nv{\mathcal{M}}_{\bullet}$ be the Rezk classification diagram for $\mathcal{M}$, and let $\widehat{\Nv{\mathcal{M}}}_{\bullet}$ be any Reedy-fibrant replacement for $\Nv{\mathcal{M}}_{\bullet}$. Then $\widehat{\Nv{\mathcal{M}}}_{\bullet}$ is a complete Segal space.
\end{cor}
\begin{proof}
By paragraph 8.1 in \citep{Dwyer-Kan:1980c}, $\mathcal{M}$ admits a homotopical three-arrow calculus, and it is well known that $\mathcal{M}$ is a saturated relative category,\footnote{See \eg, Theorem 1.2.10 in \citep{Hovey:1999} or Theorem 8.3.10 in \citep{Hirschhorn:2003}.} so we may apply \autoref{thm:homotopical.three-arrow.calculi.and.the.Rezk.classification.diagram}.
\end{proof}

\begin{remark}
Let $\mathcal{C}$ be a relative category and let $\mathcal{C}^{\natural}$ be $\und \mathcal{C}$ considered as a relative category where the weak equivalences are the isomorphisms, and let $\widehat{\Nv{\mathcal{C}}}_{\bullet}$ be a fibrant replacement for $\Nv{\mathcal{C}}_{\bullet}$ \emph{in the model structure for complete Segal spaces}. Then $\widehat{\Nv{\mathcal{C}}}_{\bullet}$ has the expected homotopy-theoretic universal property with respect to complete Segal spaces, namely:
\begin{itemize}
\item For each complete Segal space $\mathcal{D}_{\bullet}$, the evident morphism $\Nv{\mathcal{C}^{\natural}}_{\bullet} \to \widehat{\Nv{\mathcal{C}}}_{\bullet}$ induces (by precomposition) a weak homotopy equivalence
\[
\ulHom{\widehat{\Nv{\mathcal{C}}}}{\mathcal{D}} \to \ulHom \sp{\prime} \argptwo{\Nv{\mathcal{C}^{\natural}}}{\mathcal{D}}
\]
where the codomain is the simplicial set of morphisms $\Nv{\mathcal{C}^{\natural}}_{\bullet} \to \mathcal{D}_{\bullet}$ that send morphisms in $\weq \mathcal{C}$ to equivalences in $\mathcal{D}_{\bullet}$.
\end{itemize}
This is true without further hypotheses on $\mathcal{C}$: see \citep{MO:answer-93139}. In view of \autoref{prop:classifying.space.of.morphisms.with.homotopical.three-arrow.calculi} and \autoref{thm:homotopical.three-arrow.calculi.and.the.Rezk.classification.diagram}, the above yields another proof of the correctness of the hammock localization of $\mathcal{C}$.
\end{remark}

\begin{remark}
It is possible for $\widehat{\Nv{\mathcal{C}}}_{\bullet}$ to be a complete Segal space \emph{without} $\mathcal{C}$ admitting a homotopical three-arrow calculus. Indeed, if every morphism in $\mathcal{C}$ is a weak equivalence, then every face and degeneracy operator of $\Nv{\mathcal{C}}_{\bullet}$ is a weak homotopy equivalence of simplicial sets (by \autoref{lem:homotopy.equivalences.of.categories}), so $\widehat{\Nv{\mathcal{C}}}_{\bullet}$ is a complete Segal space. (In fact, up to Reedy weak equivalence, every complete Segal space in which every morphism is invertible arises in this way.) On the other hand, $\mathcal{C}$ could be a relative category in which there is no upper bound to the length of zigzags needed to represent morphisms in $\Ho \mathcal{C}$, such as the relative category generated by the infinite graph of the form below,
\[
\hspace{-0.5in}
\begin{tikzcd}
\cdots \rar &
\bullet \rar[leftarrow] &
\bullet \rar &
\bullet \rar[leftarrow] &
\bullet \rar &
\bullet \rar[leftarrow] &
\bullet \rar &
\bullet \rar[leftarrow] &
\cdots
\end{tikzcd}
\hspace{-0.5in}
\]
with all morphisms being weak equivalences.

It is also possible for $\widehat{\Nv{\mathcal{C}}}_{\bullet}$ to be a complete Segal space when the weak equivalences in $\mathcal{C}$ do \emph{not} have the 2-out-of-3 property. Indeed, the (coun\-ter)\-ex\-am\-ple considered in \autoref{rem:2-out-of-3.and.the.homotopical.three-arrow.calculus} has the property that $\Nv{\mathcal{C}}_{\bullet}$ is levelwise contractible, so $\widehat{\Nv{\mathcal{C}}}_{\bullet}$ is trivially a complete Segal space.

Other than the observation made in the introduction, we know of no necessary conditions.
\end{remark}

\ifdraftdoc

\else
  \printbibliography
\fi

\end{document}